\documentclass{article}
\usepackage{amssymb, mathrsfs, amsmath, amsthm, graphicx}
\usepackage[inline]{enumitem}
\usepackage{hyperref, tikz}
\usepackage{comment,color,todonotes}
\usepackage{tikz-cd}
\usepackage[capitalise]{cleveref}
\usepackage[margin=1.6in]{geometry}

\DeclareMathOperator{\lk}{Lk}

\newtheorem{thm}{Theorem}[section]
\newtheorem{proposition}[thm]{Proposition}
\newtheorem{lem}[thm]{Lemma}
\newtheorem{claim}[thm]{Claim}

\newtheorem*{thm*}{Theorem}

\newtheorem{innercustomthm}{Theorem}
\newenvironment{customthm}[1]
{\renewcommand\theinnercustomthm{#1}\innercustomthm}
{\endinnercustomthm}

\newtheorem{prop}[thm]{Proposition}

\theoremstyle{definition}
\newtheorem{defn}[thm]{Definition}

\newtheorem{remark}[thm]{Remark}

\author{Michael Ben-Zvi, Robert Kropholler and Rylee Alanza Lyman}
\title{Folding-like techniques for CAT(0) cube complexes}

\begin{document}
\maketitle
\begin{abstract}
In a seminal paper, Stallings introduced folding of morphisms of graphs.
One consequence of folding is the representation of finitely-generated subgroups 
of a finite-rank free group
as immersions of finite graphs.
Stallings's methods allow one to construct this representation algorithmically,
giving effective, algorithmic answers and proofs to classical questions about
subgroups of free groups.
Recently Dani--Levcovitz used Stallings-like methods to study subgroups 
of right-angled Coxeter groups, which act geometrically on CAT(0) cube complexes.
In this paper we extend their techniques 
to fundamental groups of non-positively curved cube complexes.
\end{abstract}
\section{Introduction}
In the seminal paper \cite{Stallings}, Stallings studies subgroups of free groups by introducing a technique called \textit{folding}. For finitely generated subgroups, Stallings's technique gives an algorithmic method for answering classical algebraic questions about the subgroup, such as answering the membership problem, determining the subgroup's rank and index, and determining whether or not it is normal.

More explicitly,
Stallings showed that if $f\colon (\Lambda,p) \to (\Gamma,q)$
is a map of basepointed graphs, then the map $f$ factors as
\[\begin{tikzcd}
	(\Lambda,p) \ar[rr,"f"] \ar[rd,"p",swap]
	& & (\Gamma,q) \\
	& (\hat\Lambda,\hat p) \ar[ur,"\hat f",swap], &
\end{tikzcd}\]
where the map $\hat f$ is an immersion (i.e. a local isometry),
and the map $p$ is surjective and can be realized as a sequence
of elementary moves called \emph{folds,}
which identify pairs of edges in the domain graph.
Moreover, if $\Lambda$ is a finite graph, 
then $\hat\Lambda$ can be computed algorithmically. The group-theoretic results from this technique follow when $\Gamma$ is a rose with $n$ petals and $f_\ast(\pi_1(\Lambda,p))$ is a finitely generated subgroup of $F_n = \pi_1(\Gamma,q)$. 

Recently, Dani--Levcovitz \cite{DaniLev} used Stallings-like methods 
to study subgroups of right-angled Coxeter groups.
Specifically, to any subgroup $H$ they associate a \emph{completion},
a non-positively curved cube complex $X_H$ equipped with a local isometry
to the \emph{Davis complex} for the right-angled Coxeter group,
which plays the role of the rose with $n$ petals.
Their techniques are combinatorial,
inspired by those of Kharlampovich, Miasnikov and Weil in \cite{KharlampovichMiasnikovWeil}.
Using Dunwoody's theory of \emph{tracks,} Beeker and Lazarovich \cite{BeekerLazarovich}
study Stallings folding for cubulated hyperbolic groups.
The purpose of this paper is to unite these common threads:
we show how to perform Stallings folding 
for fundamental groups of non-positively curved cube complexes.
We relate the completions of Dani--Levcovitz to the geometry of the cube complex in question,
thus proving a uniqueness result left as a question in \cite{DaniLev}.

\begin{thm}\label{mainthm}
	Let $(Y,q)$ be a basepointed non-positively curved cube complex.
	Let $H$ be a subgroup of $\pi_1(Y, q)$ with generating set $S$ given as a based map $\iota (\Gamma_S, p) \to (Y, q)$, where $\Gamma_S$ is a graph. 
	Then, there exists a unique pointed non-positively curved cube complex $(X_H,q)$
	and a local isometry $(X_H,q) \to (Y,q)$ 
	whose image on fundamental groups is $H$ and through which $f$ factors.
	
	Moreover, there is an algorithm which takes as input $H$ and $S$
	and produces (in the limit) $X_H$.
	If $S$ is finite and $H$ is quasi-convex 
	(i.e. convex cocompact on the universal cover $\tilde Y$),
	then $X_H$ is compact and the algorithm terminates in finite time.
\end{thm}

We use this algorithm to prove several results about quasi-convex subgroups of cubical groups. Namely, we prove the following, 

\begin{customthm}{\ref{finalthm}}
	Let $G = \pi_1(Y,q)$ for a nonpositively-curved cube complex $Y$.
	Let $H$ be a quasiconvex subgroup of $G$.
	Then there is an algorithm that solves the following problems.
	\begin{enumerate}
		\item The membership problem for $H$.
		\item Whether, given $g\in G$, there exists an $n$ such that $g^n\in H$. 
		\item Whether $H$ is normal in $G$.
		\item Whether $H$ is of finite index $G$, and determines the index if so.
	\end{enumerate}
\end{customthm}

The algorithm is essentially identical to the extension of Stallings's folding
presented in \cite{DaniLev}: 
For $\Gamma_S$ to reflect properties of $H$,
one wants $\iota$ to be a local isometry.
To accomplish this, one alters $\Gamma_S$ by means of three moves:
\begin{enumerate}
	\item \emph{Folding,} Stallings's original move, which identifies pairs of edges
		in the $1$-skeleton of the domain cube complex,
	\item \emph{Cube Attachment,} which fills in cubes ``missing'' from the domain
		cube complex, and
	\item \emph{Cube Identification,} which acts as a kind of higher-dimensional
		generalization of folding.
\end{enumerate}

The uniqueness result stems from the following fact.
Covering space theory associates to the map $\iota\colon (\Gamma_S,\ast) \to (Y,q)$
a map $\tilde\iota\colon (\tilde\Gamma_S,\tilde\ast) \to (\tilde Y,\tilde q)$
of universal covers.
In the case that $\iota$ is a local isometry,
the map $\tilde \iota$ will be an isometric embedding.
Using the same notation as above, we have the following result.

\begin{thm}\label{cubicalconvexhull}
	Let $H$ and $S$ be as in \cref{mainthm}. 
	If $f\colon (X_H,p) \to (Y,q)$ is the local isometry constructed in
	\Cref{mainthm}, the image of the isometric embedding
	$\tilde f\colon (\tilde X_H,\tilde p) \to (\tilde Y,\tilde q)$
	is the cubical convex hull of the lift $\tilde{\iota}\colon T\to \tilde{Y}$, where $T$ is the universal cover of $\Gamma_S$.
\end{thm}

\section*{Acknowledgements}
The second author was funded by the Deutsche Forschungsgemeinschaft (DFG, German Research Foundation) under Germany's Excellence Strategy EXC 2044 –390685587, Mathematics Münster: Dynamics–Geometry–Structure.

\section{A Cubical Completion}

Recall that a \emph{cube complex} is a metric cell complex constructed from Euclidean
unit cubes $[-1,1]^n$ glued together by isometries of their faces.

The local structure of a cube complex $X$ near a vertex $v$ is captured by its \emph{link,}
$\lk_X(v)$. The link may be thought of as a small sphere around the vertex.
As such, combinatorially $\lk_X(v)$ has one $n$-simplex for each $(n+1)$-cube containing $v$,
and the faces of such an $n$-simplex correspond to the faces of the $(n+1)$-cube
which also contain the vertex $v$.
For example, the link of a vertex in a $3$-cube is a $2$-simplex.
A cube complex is \emph{non-positively curved}---i.e. the induced path metric
locally satisfies the CAT(0) condition---if the link of each vertex is a \emph{flag simplicial complex}
\cite{Sageev,Gromov}; a higher-dimensional simplex is present if and only if its $1$-skeleton is.\footnote{
Warren Dicks: ``every non-simplex contains a non-edge."} 
A simply connected, non-positively curved cube complex is \emph{globally} CAT(0).
In particular, it is contractible.
A \emph{CAT(0) cube complex} is a simply connected non-positively curved cube complex.

Let $X$ and $Y$ be non-positively curved cube complexes,
and let $f\colon X \to Y$ be a \emph{cubical map}---that is,
$f$ sends $n$-cubes of $X$ isometrically to $n$-cubes of $Y$.
Then for each vertex $v$ of $X$, $f$ induces a map
$f_v \colon \lk_X(v) \to \lk_Y(f(v))$.
We say that the map $f$ is an \emph{immersion} if each induced map $f_v$ is injective.
If additionally the image of each $f_v$ is a \emph{full} subcomplex of $\lk_Y(f(v))$,
we say that $f$ is a \emph{local isometry.}
We have the following basic fact \cite[Prop.~4.14]{BH99}.

\begin{prop}\label{inducedmaponCC}
	Let $f\colon(X,p) \to (Y,q)$ be a map of basepointed,
	non-positively curved cube complexes which is a local isometry.
	Then the induced map $f_\ast\colon\pi_1(X,p) \to \pi_1(Y,q)$
	is injective.
	
	If additionally $f$ is surjective and 
	each local map $f_v \colon \lk_X(v) \to \lk_Y(f(v))$ is an isomorphism, 
	then $f$ is a covering map.
\end{prop}

Consequently, if in \cref{inducedmaponCC}, $Y$ is a CAT(0) cube complex
(and so simply connected),
then so is $X$, and the map $f$ is an isometric embedding,
provided that $X$ is connected.

Let $(Y,q)$ be a compact, connected, basepointed non-positively curved cube complex
and $G = \pi_1(Y,q)$.
In this section we describe a method that takes as input a
subgroup $H \le G$ and a generating set $S$ for $H$ represented as a set of loops in $Y^{(1)}$ based at $q$,
and produces a local isometry of non-positively curved cube complexes
$f\colon (X,p) \to (Y,q)$ such that $f_\ast(\pi_1(X,p)) = H$.
In the case that $H$ is a quasi-convex subgroup and $S$ is a finite generating set,
then the method is an algorithm which terminates in finite time.

Let $|S| = n$. 
Let $X$ be a wedge of $n$ circles with basepoint $b$. 
We can obtain a map of $X$ into $Y$ 
whose image on fundamental groups is $H$.
Subdivide each edge of the graph so that the resuling map is cubical.
At this stage we have a graph $(X,p)$ with a cubical map
$f\colon (X,p) \to (Y,q)$ such that $f_\ast(\pi_1(X,p)) = H$.

The idea, as in \cite{DaniLev}, is to progressively alter $X$
so that the map $f\colon(X,p) \to (Y,q)$ becomes closer and closer
to a local isometry.
In the direct limit, a local isometry is always achieved,
but if $H$ is assumed to be quasi-convex, then a finite number of operations suffice.

\paragraph{Folding}
Suppose first that the restriction of $f$ to the $1$-skeleton of $X$ is not an immersion.
If this is the case, then there are a pair of edges $e_1$ and $e_2$
with a common initial vertex $v$ such that $f(e_1) = f(e_2)$.
The map $f\colon (X,p) \to (Y,q)$ factors through the quotient map
identifying $e_1$ with $e_2$.
Replacing $X$ with this quotient cube complex is the operation of \emph{folding.}
By repeatedly folding when necessary, 
we may replace any cubical map $f\colon (X,p) \to (Y,q)$
with a map which restricts to an immersion of the $1$-skeleton of $X$ into $Y$.
Note that because folding decreases the number of edges of $X$,
if $X$ is compact, only finitely many folds are possible.

\paragraph{Cube Identification}
Suppose that in $X$, there are a pair of $n$-cubes
$c_1$ and $c_2$ whose $1$-skeleta are equal.
Since $Y$ is non-positively curved,
we conclude that $f(c_1) = f(c_2)$.
In order for $f$ to be a local isometry, we require $c_1 = c_2$.
Therefore we replace $X$ by the quotient cube complex obtained by identifying $c_1$ with $c_2$.
This process is \emph{cube identification;} 
one may think of it as a higher-dimensional analogue of folding.
As with folding, the number of cubes of $X$ decreases,
so if $X$ is compact, only finitely many cube identifications are possible.

In particular, by repeatedly folding and then repeatedly performing cube identifications,
we may assume that for each vertex $v \in X$,
the local map $f_v \colon \lk_X(v) \to \lk_Y(f(v))$ is injective.
However, for $f$ to be a local isometry,
we need the image $f_v(\lk_X(v))$ to contain each simplex it spans.
This leads us to our final operation.

\paragraph{Cube Attachment}
Suppose that $f$ fails to be a local isometry because the image of $f_v$
is not a full subcomplex. That is, there is a set of $n+1$ vertices of $\lk_X(v)$
whose image under $f_v$ span a $n$-simplex in $\lk_Y(f(v))$ 
but do not span a $n$-simplex in $\lk_X(v)$.
These vertices correspond to edges $e_1,\dotsc,e_{n+1}$ incident to $v$
whose image in $Y$ are the $n+1$ edges contained in an $(n+1)$-cube of $Y$
incident to a common vertex $f(v)$.
Call this $(n+1)$-cube $c$---note that it is determined 
by the edges $f(e_1),\dotsc,f(e_{n+1})$.
To remedy the situation in $X$
we attach an $(n+1)$-cube $c'$ to $v$
and the edges $e_1,\dotsc,e_{n+1}$.
The map $f$ extends uniquely to $c'$ by mapping $c'$ to $c$
preserving the images of the edges $e_1,\dotsc,e_{n+1}$.
This operation is \emph{cube attachment}.
If $X$ is compact (or more generally locally finite)
then only finitely many cube attachments may be performed at each vertex.

\paragraph{}
Observe that in performing each operation on $(X,q)$, there is a resulting
natural cubical map from the ``old'' cube complex to the ``new'' cube complex.
Thus by performing a sequence of foldings, cube identifications and cube attachments,
we obtain a sequence of basepointed cube complexes 
$(X,p) = (X_1,p_1) \to (X_2,p_2) \to \dotsb$. 
Let $(\hat X,\hat q)$ denote the direct limit of this sequence.
By the universal property of the direct limit,
$(\hat X,\hat p)$ is a cube complex equipped with a cubical map
$\hat f\colon (\hat X,\hat p) \to (Y,q)$.
We say that $\hat X$ is a \emph{completion} of $X$ if the map $\hat f$ is a local isometry.

\begin{thm}\label{completiondownstairs}
	Let $f\colon (X,p) \to (Y,q)$ be a cubical map, where $Y$ is compact
	and non-positively curved and $X$ is connected.
	There exists a completion of $X$.
\end{thm}

\begin{proof}
	We are interested in the case where $(X,p)$ is compact, so we shall assume $(X,p)$ is compact.
	The general case presents no real new difficulties, 
	except that one may have to pass to direct limits multiple times in the process
	of constructing a completion.
	Consider the sequence
	\[
		\begin{tikzcd}
			(X,p) = (X_1,p_1) \ar[r] & (X_2,p_2) \ar[r] & (X_3,p_3) \ar[r] & \dotsb
		\end{tikzcd}
	\]
	defined by performing the following procedure:
	Begin by performing all possible folds,
	then perform all possible cube identifications.
	If $(X_{k+1},p_{k+1})$ is obtained from $(X_k,p_k)$
	by attaching a cube, then $(X_k,p_k)$ is naturally a subcomplex of $(X_{k+1},p_{k+1})$.
	If $(X_k,p_k)$ admits no folds or cube identifications,
	we thus perform all possible cube attachments at the various vertices of $(X_k,p_k)$.
	This done, we repeat the process beginning with folding.

	Let $(\hat X,\hat p)$ denote the direct limit of this sequence.
	We claim that $(\hat X,\hat p)$ is a completion of $(X,p)$.
	Suppose $c_1$ and $c_2$ are cubes incident to a vertex $v$,
	and that $\hat f(c_1) = \hat f(c_2)$.
	Then there exists $k$ such that we may identify $c_1$ and $c_2$
	as belonging to $X_k$, and we must have that $f_k\colon (X_k,p_k) \to (Y,q)$
	satisfies $f_k(c_1) = f_k(c_2)$.
	This implies that at some later stage $k'$, after performing foldings and cube identifications,
	the images of $c_1$ and $c_2$ in $X_{k'}$ are equal.
	Thus we see that each $\hat f_v$ is injective.

	To see that its image is a full subcomplex, we argue similarly:
	for every collection of edges $e_1,\dotsc e_{n+1}$ incident to $v$
	whose image under $\hat f_v$ span an $n$-simplex of $\lk_Y(\hat f(v))$,
	there exists $k$ such that we may identify $e_1,\dotsc, e_{n+1}$ as edges of $X_k$
	incident to $v$.
	But this implies that at some later stage $k'$,
	possibly after performing cube attachments, there exists an $(n+1)$-cube $c$ in $X_{k+1}$
	having $e_1,\dotsc,e_{n+1}$ as faces.
	This implies that in $\lk_X(v)$ in $\hat X$, the vertices corresponding to $e_1,\dotsc, e_{n+1}$
	span an $n$-simplex, and the process of cube identification ensures that this $n$-simplex is unique.
	Thus we see that $\hat f$ is a local isometry.
\end{proof}

We will refer to the completion obtained by the method described in the proof as the \emph{standard}
completion of $X$. Later we will show that this completion coincides with any completion of $X$.

\begin{prop}
	Let $G$ be a right-angled Coxter group. Let $H \le G$ be a finitely generated
	torsion-free subgroup, and let $X$ be a boquet of circles in one-to-one correspondence
	with the generators of $H$. Then the completion of $\hat X$ described above
	is the same as the standard completion $\Omega$ from \cite{DaniLev}.
\end{prop}

\begin{proof}
	In \cite{DaniLev}, there are three operations applied to obtain a completion.
	There, the role of $Y$ is played by the quotient of the \emph{Davis complex} $\tilde Y$ by the action of $G$.
	Because $G$ contains torsion, $G$ is not the usual fundamental group of $Y$.
	Nevertheless, because $G$ acts freely on the vertices of $\tilde Y$
	there is a well-defined local map $f_v$ for each vertex 
	of $X$, and we may use the data of this local map to perform the completion as in our theorem.
	In the absence of elements of order two, these two methods are identical.
\end{proof}

\begin{remark}
	A more sophisticated  approach to the above proposition might consider the quotient of the Davis complex
	as a \emph{complex of groups,}
	and perform our operations on a map $f\colon(X,p) \to (Y,q)$ of complexes of groups.
	This approach could allow one to consider arbitrary cubulated groups
	and subgroups with torsion.
  \end{remark}

\section{Constructing Completions in the Universal Cover}

We now present another look at the completion,
using the hyperplanes of the universal cover to construct a completion.

Let us recall: a \emph{midcube} in an $n$-cube $[-1,1]^n$ is the set of points
with $k$th coordinate $0$ for some fixed $k$ satisfying $1 \le k \le n$.
Let $X$ be a cube complex.
We say two edges of $X$, $e$ and $e'$ are \emph{related}
if they intersect a common midcube, and let $\sim$ denote the equivalence relation
on the set of edges generated by the above relation.
The equivalence classes of $\sim$ are the \emph{hyperplanes} of $X$.
By gluing together the corresponding common midcubes,
one may think of each hyperplane of $X$ as a (possibly immersed) subspace $\mathfrak{h} \subset X$.
We say a hyperplane is \emph{dual} to each edge in the corresponding $\sim$-equivalence class.
In a CAT(0) cube complex $X$, hyperplanes are convexly embedded; 
each hyperplane $\mathfrak{h}$ separates $X$ into two components, $H^+$ and $H^-$,
called \emph{half-spaces},
and in each cube $\mathfrak{h}$ meets, $\mathfrak{h}$ is a subspace of codimension one. 

Hyperplanes turn out to be key to the study of CAT(0) cube complexes.
In particular, if $X$ is a CAT(0) cube complex,
consider the path metric on its $1$-skeleton;
this is the \emph{combinatorial} or \emph{$\ell_1$ metric.}
A subcomplex $Y$ of $X$ is \emph{cubically convex}
if 
\begin{enumerate}
	\item Given vertices $v$ and $w$ of $Y$,
		the $1$-skeleton of $Y$ contains every combinatorial geodesic (shortest path)
		from $v$ to $w$.
	\item If the $1$-skeleton of an $n$-cube is contained in $Y$,
		then the $n$-cube is contained in $Y$.
\end{enumerate}
If $S$ is a subset of $X$, its \emph{cubical convex hull}
is the smallest cubically convex subcomplex of $X$ containing $S$.
Let $\gamma$ be a combinatorial geodesic between vertices $v$ and $w$. 
Since hyperplanes are separating we see that $\gamma$ crosses every hyperplane which separates $v$ from $w$. 
Moreover, a path $\zeta$ in the 1-skeleton is a combinatorial geodesic $v$ to $w$ if and only the hyperplanes $\zeta$ crosses are exactly those separating $v$ from $w$ and each is crossed exactly once. 

Let $f\colon(X,p) \to (Y,q)$ be a basepointed cubical map.
Let $\tilde X$ and $\tilde Y$ denote the universal covers of $X$ and $Y$, respectively.
The map $f$ induces a map of universal covers
$\tilde f \colon (\tilde X,\tilde p) \to (\tilde Y,\tilde q)$.
We shall use the hyperplanes of $\tilde Y$ to construct a completion of $X$.

Here is the idea:
if the map $f$ were a local isometry, then the map $\tilde f$ would recognize
$\tilde X$ as a cubically convex subcomplex of $\tilde Y$.
If not, we will essentially replace $\tilde X$ with the cubical convex hull of its image.

\paragraph{A Completion via Sageev's Construction}
Let $\mathfrak{h}$ be a hyperplane of $\tilde Y$. If $\mathfrak{h} \cap \tilde f(\tilde X)$ is nonempty,
then both of the corresponding half-spaces
$H^+$ and $H^-$ have nonempty intersection with $\tilde f(\tilde X)$.
The collection of half-spaces of $\tilde Y$ which meet $\tilde f(\tilde X)$
form a \emph{poset with complementation} see \cite[p. 11]{Sageev} or \cite{Roller}.
The poset structure is given by inclusion (in $\tilde Y$), 
and the order-reversing involution is given by the complementation $H^+ \leftrightarrow H^-$.
Call this collection $\mathscr{H}$.
It is a sub-poset of the collection of all half-spaces of $\tilde Y$.
Let $Z$ be the \emph{dual CAT(0) cube complex} obtained by applying
Sageev's construction \cite{SageevEnds} to $\mathscr{H}$.
Recall that this means $Z$ has the following description.
A vertex $v$ of $Z$ is a consistent choice of orientation (or half-space)
for each hyperplane in $\mathscr{H}$.
That is, $v$ is a subset of $\mathscr{H}$ satisfying the following three conditions.
\begin{enumerate}
	\item[\textbf{(Choice)}] For each hyperplane $\mathfrak{h}$ such that $H^+$ and $H^-$ belong to $\mathscr{H}$,
		exactly one of $H^+$ or $H^-$ belongs to $v$.
	\item[\textbf{(Upward Closed)}] If $H$ and $H'$ are half-spaces satisfying $H \subset H'$
		such that $H$ belongs to $v$, then $H'$ belongs to $v$.
	\item[\textbf{(Descending Chains)}] If $H_0 \supset H_1 \supset H_2 \supset \dotsb$
		is a descending chain of half-spaces belonging to $v$,
		then this chain \emph{terminates}---there is some integer $N$
		such that $H_n = H_N$ for all $n \ge N$.
\end{enumerate}

The key example to keep in mind is the collection of all half-spaces in a CAT(0) cube complex
containing a given vertex:
the reader may wish to verify that this collection satisfies the three conditions above.

An \emph{edge} of $Z$ connects two vertices $v$ and $w$ when they differ by a single choice of halfspace.
This means that there is some hyperplane $\mathfrak{h}$ such that one half-space $H^+$ belongs to $v$, say,
while $H^-$ belongs to $w$---otherwise a half-space belongs to $v$ if and only if it belongs to $w$.
This defines the $1$-skeleton of $Z$; 
to finish we declare that a higher-dimensional cube is present if and only if its $1$-skeleton is.
The result is a CAT(0) cube complex (see \cite[Section 2]{Sageev}).
The half-spaces in $\mathscr{H}$ are canonically identified with the half-spaces 
associated to the hyperplanes of $Z$.
This duality between posets with complementation and CAT(0) cube complexes goes by the name \emph{Roller duality;}
we shall make light use of it in what follows.

\paragraph{The Completion is the Cubical Convex Hull} 
The CAT(0) cube complex $Z$ comes equipped with a map $\hat f\colon Z \to \tilde Y$
defined in the following way.
Consider the collection of half-spaces of $\tilde Y$ containing the basepoint $\tilde q$.
Since $\tilde q$ is in the image of $\tilde f$,
the restriction of this collection to $\mathscr{H}$ defines a vertex of $Z$.
Call this vertex $\hat p$ and define $f(\hat p) = \tilde q$.
The collection of half-spaces of $Z$ is $\mathscr{H}$, 
which is a subset of the collection of half-spaces of $\tilde Y$.
Let $z$ be a vertex of $Z$.
We claim that there exists a unique vertex $\hat f(z)$ of $\tilde Y$
in the intersection of the following collection of half-spaces of $\tilde Y$
\[ \{ H \in \mathscr{H} \mid H \text{ belongs to } z \} \cup \{ H \notin\mathscr{H} \mid \tilde q \in H \}. \]

To see this, it suffices to show that this collection of half-spaces satisfies the upward-closed condition.
So suppose $H$ and $H'$ are half-spaces satisfying $H \subset H'$,
and suppose $H$ belongs to the above collection.
Recall that $H$ belongs to $\mathscr{H}$ if and only if $H \cap \tilde f(\tilde X)$ is nonempty.
Therefore if we have $H \in \mathscr{H}$, then $H'$ is in $\mathscr{H}$ as well,
and we conclude that $H'$ belongs to $z$ as required.
On the other hand, if $H$ does not belong to $\mathscr{H}$,
then its complement has empty intersection with $\tilde f(\tilde X)$,
as does the complement of $H'$. 
Therefore $H'$ does not belong to $\mathscr{H}$ but does contain $\tilde q$, as required.

\begin{proposition}\label{completionupstairs}
	The map $\hat f \colon (Z,\hat p) \to (\tilde Y,\tilde q)$ is an isometric embedding.
	Its image is the cubical convex hull of $\tilde f(\tilde X)$ in $\tilde Y$.
\end{proposition}

\begin{proof}
	To prove the former statement, it suffices to show that $\hat f$ is a local isometry.
	So let $v \in Z$ be a vertex. We want to show that the map
	$\hat f_v\colon \lk_Z(v) \to \lk_{\tilde Y}(\hat f(v))$ is injective,
	and that its image is a full subcomplex of $\lk_{\tilde Y}(\hat f(v))$.
	Indeed, vertices of $\lk_Z(v)$ correspond to edges dual to hyperplanes of $\mathscr{H}$.
	The map $\hat f$ takes hyperplanes in $\mathscr{H}$ to themselves, so the map is injective.
	A set of vertices spans a simplex in $\lk_{\tilde Y}(\hat f(v))$ by definition 
	if and only if they span a simplex in $\lk_Z(v)$, so the image is a full subcomplex.

	Since the map $\hat f$ is an isometric embedding, its image is cubically convex in $\tilde Y$
	and contains $f(X)$ by definition.
	Conversely, any combinatorial geodesic between points of $f(X)$ 
	crosses only hyperplanes in $\mathscr{H}$,
	so is contained in the image of $\hat f$.
\end{proof}

\section{Equivalences}
The main result of this section is the following theorem, which refines \cref{cubicalconvexhull}.

\begin{thm}
	Let $f\colon (X,p) \to (Y,q)$ be a basepointed cubical map, and write
	\linebreak $H = f_\ast(\pi_1(X,p)) \le \pi_1(Y,q)$.
	Let $(\hat X,\hat p)$ denote the completion in \cref{completiondownstairs}
	with local isometry $\hat f\colon (\hat X,\hat p) \to (Y,q)$.
	and let $Z$ denote the cube complex obtained from $\tilde f(\tilde X)$
	in \cref{completionupstairs} isometrically embedded in $\tilde Y$.

	Then $Z/H = \hat X$, or equivalently $Z$ is the universal cover of $\hat X$.
\end{thm}

\begin{proof}
	Let $(U,\tilde p)$ denote the universal cover of $(\hat X,\hat p)$.
	Because the map $\hat f$ is a local isometry, its lift to an isometric embedding
	$\tilde f \colon (U,\tilde p) \to (\tilde Y,\tilde q)$.
	The image of $\tilde f$ contains the image of the universal cover $\tilde X$ of $X$, by definition,
	and is a convex subcomplex of $\tilde Y$.
	Since $Z$ is the cubical convex hull of the image of $\tilde X$ in $\tilde Y$,
	we have $Z\subset \tilde f(U)$.
	On the other hand, using Roller duality, we see that $U$ is contained in the cube complex
	obtained from the collection of half-spaces intersecting the image of $\tilde X$.
	That is, $\tilde f(U) \subset Z$, so we conclude $\tilde f(U) = Z$.
\end{proof}

We can use the previous theorem to prove the uniqueness of the completion obtained in \cref{completiondownstairs}.

\begin{thm}
	Let $Z$ and $Z'$ be completions of $f\colon (X,p) \to (Y,q)$.
	Then $Z = Z'$.
\end{thm}

\begin{proof}
	Let $\tilde Z$ and $\tilde Z'$ denote the universal covers of $Z$ and $Z'$,
	thought of as cubically convex subcomplexes of $\tilde Y$.
	Note that every hyperplane of $\tilde Z$ and $\tilde Z'$ intersects
	$\tilde f(\tilde X)$. Thus we conclude that $\tilde Z$ and $\tilde Z'$
	are both the convex hull of $\tilde f(\tilde X)$ in $\tilde Y$.

	Since $\pi_1(Z) = \pi_1(Z')$, and the cube complexes $Z$ and $Z'$
	have the same universal cover, we conclude $Z = Z'$.
\end{proof}

\section{Consequences}

In this section, we collect group theoretic results following from our results on completions.
First let us complete the proof of \cref{mainthm}, which follows from the following theorem.
Throughout this section let $G = \pi_1(Y,q)$.

\begin{thm}\label{quasiconvexsubgroup}
	Let $T$ be a maximal tree in the $1$-skeleton of the finite nonpositively-curved cube complex $(Y,q)$. 
	Let $S$ be the set of edges not in $T$. 
	Thus $S$ is a generating set for $G$. 
	Let $H$ be a subgroup of $G$ which is quasiconvex with respect to the generating set $S$,
	and let $(X,p)$ be a finite graph with a map $f\colon (X,p) \to (Y,q)$ realizing
	the inclusion $H \hookrightarrow G$.
	Then the completion of $X$ is a finite cube complex.
\end{thm}

To complete the proof of \cref{mainthm} assuming \cref{quasiconvexsubgroup}, note that because
the completion of $X$ is finite, the algorithm outlined in \cref{completiondownstairs}
will terminate in finite time.

Conversely, note also that if some completion of a realization of $H$ is finite,
then $H$ is a quasiconvex subgroup---thus quasiconvexity on fundamental groups may be recognized
by topological properties of the completion.

\begin{proof}[Proof of \cref{quasiconvexsubgroup}]
	The image of $\tilde f(\tilde X)$ in $\tilde Y$ is a quasiconvex subset.
	To see this, note that collapsing the preimage of the maximal tree in $\tilde Y$
	yields a complex whose $1$-skeleton is the Cayley graph of $G$ with respect to $S$,
	and that the image of $\tilde f(\tilde X)$ under this collapse contains the vertices
	belonging to $H$.

	From \cite[Theorem H]{Haglund} we see that the convex hull of $\tilde f(\tilde X)$ has a
	cocompact $H$-action whose quotient is a completion of $X$. 
	Thus the completion of $X$ is finite.
\end{proof}

Another property recognized by the completion is the property of $H$ having finite index in $G$.
In \cite{DaniLev}, some care needs to be taken in the case that the right-angled Coxeter group
has a nontrivial finite central subgroup. 
This happens exactly when 
the action of the right-angled Coxeter group on its Davis complex is not \emph{essential.}
Recall that if $G$ acts properly and cocompactly on a CAT(0) cube complex $\tilde Y$,
then there exists a subcomplex on which the action is essential.

\begin{thm}
	Let $G$ be a group acting freely, properly, cocompactly and 
	essentially\footnote{O'Donnell calls this kind of action a ``cube-awesome" action.} 
	on a CAT(0) cube complex $\tilde Y$, and let $Y = \tilde Y/G$.
	Let $H$ be a subgroup of $G$ and let $Z$ denote a corresponding completion of $H$.
	Then $H$ has finite index in $G$ if and only if $Z$ is a finite cover of $Y$.
\end{thm}

\begin{proof}
	The assertion which needs proof is that if $H$ has finite index in $G$,
	then the completion $Z$ is a finite cover of $Y$.
	Let $\mathfrak{h}^+$ and $\mathfrak{h}^{-}$ be complementary half-spaces of $\tilde Y$.
	Since $G$ acts essentially on $\tilde Y$ and $H$ has finite index,
	there is an element $g \in H$ which \emph{skewers} $\mathfrak{h}^+$,
	i.e. the action of $g$ satisfies 
	$g.\mathfrak{h}^+ \subset \mathfrak{h}^+$ and $g^{-1}.\mathfrak{h}^{-} \subset \mathfrak{h}^-$.
	This implies that if $v \in \tilde Y$ is a vertex,
	there exists a power $k$ such that $g^k.v \in \mathfrak{h}^+$ and $g^{-k}.v \in \mathfrak{h}^-$.
	This implies the corresponding hyperplane belongs to the collection of hyperplanes
	used in \cref{completionupstairs} to construct the universal cover of $Z$.
	Thus we conclude that the universal cover of $Z$ is all of $\tilde Y$.
	Therefore $Z = \tilde Y/H$ is a finite cover of $Y$.
\end{proof}

We now move to a characterization of $H$ being normal in $G$.
As in \cite{DaniLev}, we require a kind of ``core graph'' for our completion.
First let us discuss normal forms for a group that is the fundamental group of a nonpositively-curved
cube complex.

\begin{lem}
	Let $Y$ be a cube complex. Let $T$ be a maximal tree in the $1$-skeleton of $Y$.
	Let $E$ be the set of edges of $Y$ and $S$ the set of squares in $Y$.
	Then $\pi_1(Y,q)$ has a presentation of the form
	\[ \pi_1(Y,q) = \langle E \mid \{e\in E\mid e \in T\} \cup \{\partial s \mid s \in S\} \rangle. \]
	The above presentation is the \emph{cubical presentation} for $\pi_1(Y,q)$.
\end{lem}

There is a natural map from the $1$-skeleton of $\tilde Y$ to the Cayley graph for
$G = \pi_1(Y,q)$ with respect to the generating set $E$ given by identifying the vertices
in each lift of $T$ to $\tilde Y$.

\begin{defn} Let $w$ be a word in the cubical presentation for $G$.
	We say that $w$ is a \emph{cubical word} if it lifts to a path in $\tilde Y$.
	
	A cubical word $w$ is \emph{reduced} if it lifts to a combinatorial geodesic in $\tilde Y$.
\end{defn}

Every element of $G$ can be represented by a cubical word. 
If a word does not lift to a path in $\tilde X$,
then there are points of discontinuity when trying to construct a lift.
We can fill these in with edges in the tree $T$.
In fact, we can make this process algorithmic which will be important for answering algorithmic questions later. 

\begin{prop}
	Let $(Y, q)$ be a non-positively curved cube complex. 
	Let $\mathcal{P} = \langle E\mid R\rangle$ be the associated cubical presentation. 
	Let $w$ be a word in $F(E)$. 
	Then there is an algorithm that computes all reduced cubical words equal to $w$. 
\end{prop}
\begin{proof}
	We begin by replacing $w$ by a cubical word. 
	We can view $w$ as a sequence of edges $(e_0, e_1, \dots, e_l)$. 
	In the case that $\iota(e_i) \neq \tau(e_{i-1})$ we can take a path in the maximal tree from $\tau(e_{i-1})$ to $\iota(e_i)$ and add this to our sequence to ensure this path is continuous. 
	It is clear that this procedure does not change $w$. 
	A similar procedure is used to ensure that the path starts and ends at $q$. 
	
	We have now represented $w$ as a loop in $Y$ based at $q$. 
	View this as a map $(I, 0)\to (Y, q)$ subdivided to be made cubical. 
	We can now take the completion of this map.
	This will be a finite cube complex as the $\pi_1(I, 0) = \{e\}$. 
	This will contain the convex hull of $\{\tilde{q}, w\cdot \tilde{q}\}$ in $\tilde{Y}$. 
	Hence, it will contain every combinatorial geodesic between $\tilde{q}$ and $w\cdot \tilde{q}$. 
	Since this complex is finite we can write down all such combinatorial geodesics. 
\end{proof}

Throughout we will freely pass between cubical words and paths in the CAT(0) cube complex
$\tilde Y$. There are two operations we can apply to cubical words that do not change
the element of $G$ represented.
\begin{enumerate}
	\item \textbf{Collapse:} remove or insert a pair $ee^{-1}$ or $e^{-1}e$ for some edge $e$ in $X$.
	\item \textbf{Square Slide:} Replace a pair $st$ with $uv$ if $stv^{-1}u^{-1}$
		is the boundary of some square in $X$.
\end{enumerate}

The proof of the following lemma follows from the simplicial approximation theorem.

\begin{lem} Suppose that $w$ and $w'$ are two cubical words that represent $g \in G$.
	Then they are related by a sequence of collapses and square slides.
	\hfill\qedsymbol
\end{lem}

The following definition comes from \cite{DaniLev}.

\begin{defn}
	Let $(Z,\hat p)$ be a completion of $f\colon (X,p) \to (Y,q)$.
	The \emph{core graph} $C(Z,\hat p)$ of $Z$ is the collection of all loops in $Z$
	giving reduced cubical words in $G$.
\end{defn}

The core graph contains all the information of the completion in the following sense.

\begin{lem}\label{coregraphsurjective}
	The inclusion $C(Z,\hat p) \to Z$ is surjective on fundamental groups.
\end{lem}

\begin{proof}
	Let $g$ be an element of $\pi_1(Z,\hat p) = H$. We can pick a cubical word $w$ representing $g$
	which lies in the image of $\tilde Z$. 
	Since $\tilde Z$ is convex in $\tilde X$, 
	there is a reduced cubical word representing $g$ in $\tilde Z$.
	This loop projects to a loop in $C(Z,\hat p)$.
\end{proof}

\begin{lem}\label{samecoresamegroup}
	Let $(Z,\hat p)$ and $(Z',\hat p')$ be completions
	of cube complexes representing $H$ in $G$.
	Then $C(Z,\hat p) = C(Z',\hat p')$.
\end{lem}

\begin{proof}
	Suppose $w$ is the label of a loop in $C(Z,\hat p)$.
	Then $w$ represents an element $g$ of $H$.
	Thus there is a loop in $(Z',\hat p')$ that represents $g$.
	Since $\tilde Z'$ is convex in $\tilde X$ every reduced path representing $g$
	is in $\tilde Z'$. Thus the path $w$ is in $\tilde Z'$.
\end{proof}

\begin{defn}
	Let $g \in G$ be a group element and $\gamma\colon [0,1] \to Y$
	be a loop in the $1$-skeleton of $Y$ based at $q$ representing $g$.
	Let $H$ be a finitely-generated subgroup of $G$
	and $(X,p)$ a cube complex representing $H$.
	Let $(X_g,p)$ be the cube complex obtained from $X$
	by attaching an interval at $p$ and extending
	$f\colon (X,p) \to (Y,q)$ by $\gamma$ on this interval.
	(Subdivide as necessary to make this a cubical map.)
	Let $x$ be the endpoint of the interval that is not $p$.
	The \emph{$g$-normalized completion} of $X$ is the
	completion $(Z_g,\hat p)$ of $X_g$. 
	Let $z$ be the image of $x$ in $Z_g$.
\end{defn}

\begin{lem}
	 Let $H$ be a subgroup of $G$ represented by a cube complex
	 $f\colon (X,p) \to (Y,q)$.
	 Let $(Z,\hat p)$ be the completion of $X$.
	 Then $H$ is normalized by $g$ if and only if $C(Z,\hat p) = C(Z_g,z)$.
 \end{lem}

 \begin{proof}
	 It is clear that $(X_g,x)$ represents the subgroup $g^{-1}Hg$.
	 Thus if $g$ normalizes $H$, then $C(Z_g,z) = C(Z,\hat p)$ by \cref{samecoresamegroup}.

	 By \cref{coregraphsurjective},
	 the core graph contains all the information about the fundamental group.
	 Thus if the completions have the same core graph, then $g^{-1}Hg = H$,
	 so $g$ normalizes $H$.
 \end{proof}

 We summarize the results of this section for quasiconvex subgroups as a theorem. For an alternative proof of the membership problem one can use biautomaticity and \cite{KharlampovichMiasnikovWeil}.
 \begin{thm}\label{finalthm}
	 Let $G = \pi_1(Y,q)$ for a nonpositively-curved cube complex $Y$.
	 Let $H$ be a quasiconvex subgroup of $G$.
	 Then there is an algorithm that solves the following problems.
	 \begin{enumerate}
		 \item The membership problem for $H$.
		 \item Whether, given $g\in G$, there exists an $n$ such that $g^n\in H$. 
		 \item Whether $H$ is normal in $G$.
		 \item Whether $H$ is of finite index $G$, and determines the index if so.
	 \end{enumerate}
 \end{thm}
\begin{proof}
	The only statement that requires proof is 2. 
	Similar to \cite{DaniLev}, we show the following 
	\begin{claim}
		A power of $g$ is in $H$ if and only if $g^k\in H$ for some $k\leq L$, where $L$ is the number of vertices of $X_H$. 
	\end{claim}
	One direction is clear, thus we focus on the forward direction. 
	Suppose that $g^n\in H$ for some $n$. 
	Then any reduced cubical representative of $g^n$ can be read in the 1-skeleton of $H$. 
	I.e. the map $[0, 1]\to Y$ representing $g^n$ factors through $H$. 
	This implies that $g\cdot \tilde{p}\in \tilde{X}_H\subset \tilde{Y}$. 
	There are $L$ orbits of vertices in $\tilde{X}_H$ thus for some $k\leq L$ we have that $g^k\cdot \tilde{p} = \tilde{p}$. 
	Thus $g^k$ gives a loop in $X_H$ based at $p$ and $g^k$ is an element of $H$. 
\end{proof}

\bibliographystyle{alpha}
\bibliography{bib.bib}
\end{document}